\documentclass[12pt]{amsart}

\usepackage{amssymb}

\newtheorem{thm}{Theorem}

\newtheorem{lem}[thm]{Lemma}

\newtheorem{que}[thm]{Question}
\newtheorem{prop}[thm]{Proposition}

\theoremstyle{definition}

\newtheorem{rem}[thm]{Remark}

\numberwithin{equation}{section}

\begin{document}

\begin{center}
{\large\bf  A Note on $\aleph_{0}$-injective Rings}

\vspace{0.8cm} { Liang Shen \\
 {\it Department of Mathematics, Southeast University
 \\ Nanjing 210096, P.R. China}\\}

 {\it E-mail: lshen@seu.edu.cn
 }\\
\end{center}

\bigskip

\noindent{\bf Abstract:} A ring $R$ is called right
$\aleph_{0}$-injective if every right homomorphism from
 a countably generated right ideal of $R$ to $R_{R}$ can be extended to a homomorphism from $R_{R}$ to $R_{R}$.
  In this note, some characterizations of $\aleph_{0}$-injective rings are given. It is proved that
  if $R$ is semiperfect, then $R$ is right $\aleph_{0}$-injective if and only  if every homomorphism from a countably
generated small right ideal of $R$ to $R_{R}$ can be extended to one
from $R_{R}$ to $R_{R}$. It is also shown that if $R$ is
  right noetherian and left $\aleph_{0}$-injective, then $R$ is \emph{QF}. This result can be looked as
   an approach to the Faith-Menal conjecture.  \\

\bigskip

 Quasi-Frobenius rings.

 \bigskip

 \noindent {\bf(2000) Mathematics Subject Classification:} \indent16D50; 16L60.

\bigskip
\section{\bf INTRODUCTION}
 \indent Throughout this paper rings are associative with
identity.  Write $J$ and $S_l$ for the Jacobson radical and the left
socle  of a ring $R$ respectively. Use $N\subseteq^{ess} M$ to mean
that $N$ is an essential submodule of $M$. For a subset $X$ of a
ring $R$, the left annihilator of $X$ in $R$ is ${\bf l}(X)=\{r\in
R: rx=0$ for all $x\in X\}$. Right annihilators are defined
analogously. $f=c\cdot $ means that $f$ is a homomorphism multiplied
by an element
$c$ on the left side.\\
\indent It is mentioned in \cite{T98} that a ring $R$ is called
right
 \emph {$\aleph_{0}$-(or countably) injective} if every right homomorphism from
 a countably generated right ideal of $R$ to $R_{R}$ can be extended to a homomorphism from $R_{R}$ to
 $R_{R}$. Recall that a ring $R$ is called right \emph{F-injective} if every right homomorphism from
 a finitely generated right ideal of $R$ to $R_{R}$ can be extended to one from $R_{R}$ to
 $R_{R}$. And a right \emph{FP-injective} ring $R$ satisfies that  every right homomorphism from
 a finitely generated submodule of a free right $R$-module $F_{R}$ to $R_{R}$ can be extended to one from $F_{R}$ to
 $R_{R}$. The left side of the above injectivities can be defined similarly. It is obvious that right self-injective rings are right
$\aleph_{0}$-injective
   and right $\aleph_{0}$-injective rings are right
 $F$-injective. But neither of the converses is true (see \cite[Example
 10.46]{T98}). The example also shows that a right \emph{FP}-injective
 ring may not be right $\aleph_{0}$-injective. But it is still unknown
 whether a right $F$-injective ring is right \emph{FP}-injective.
We have the following arrow diagrams on injectivities of rings:
\begin{center}
{\footnotesize right self-injectivity $^{\Rightarrow}_{\nLeftarrow}$
right $\aleph_{0}$-injectivity $^{\Rightarrow}_{\nLeftarrow}$ right
 $F$-injectivity $^{?\Rightarrow}_{\Leftarrow}$ right
 \emph{FP}-injectivity,
 }\end{center}
\begin{center}{\footnotesize \indent right self-injectivity
$^{\Rightarrow}_{\nLeftarrow}$ right \emph{FP}-injectivity
$^{\nRightarrow}_{\Leftarrow?}$ right
 $\aleph_{0}$-injectivity.   }\end{center}
 Recall that a ring $R$ is \emph{quasi-Frobenius} (\emph{QF})
if $R$ is one-sided noetherian and one-sided self-injective. There
are three unresolved Faith conjectures on \emph{QF} rings (see
\cite{NY03}). One of them is the Faith-Menal conjecture, which was
raised by Faith and Menal in \cite{FM94}. The conjecture says that
every strongly right Johns ring is \emph{QF}. A ring $R$ is called
 \emph{right Johns} if $R$ is right noetherian  and every right ideal of $R$ is
 a right annihilator.  $R$ is called \emph{strongly right Johns} if the
 matrix ring M$_{n}$($R$) is right Johns for all  $n\geq 1$. In \cite{J77}, Johns used a
false result of Kurshan \cite[Theorem 3.3]{K70} to show that right
Johns rings are right artinian. Later in \cite{FM92}, Faith and
Menal gave a counter example to show that right Johns rings may not
be right artinian. They characterized strongly right Johns rings as
right noetherian and left \emph{FP}-injective rings (see
\cite[Theorem 1.1]{FM94}). So the Faith-Menal conjecture is
equivalent to say that every right noetherian and left
\emph{FP}-injective ring is \emph{QF}. In this short article, some
characterizations of $\aleph_{0}$-injective rings are explored. It
is proved in Theorem 9 that if $R$ is semiperfect, then $R$  is
right $\aleph_{0}$-injective if and only  if every homomorphism from
a countably generated small right ideal of $R$ to $R_{R}$ can be
extended to one from $R_{R}$ to $R_{R}$. Since \emph{FP-}injectivity
implies $F$-injectivity, it is unknown whether a right noetherian
and left $F$-injective ring is \emph{QF}. It is proved in Theorem 14
that a right noetherian and left $\aleph_{0}$-injective ring is
\emph{QF}. This result can be looked as an approach to the
Faith-Menal conjecture.

\bigskip

\section{\bf RESULTS}

\medskip

First we explore some basic characterizations of
$\aleph_{0}$-injective rings.
\begin{prop}\label{prop 1}
A direct product of rings R = $\prod_{i\in I} R_{i}$ is right
$\aleph_{0}$-injective if and only if $ R_{i}$ is right
$\aleph_{0}$-injective, $\forall i\in I$.
\end{prop}
\begin{proof}

For $i\in I$, let $\pi_{i}$ and $\iota_{i}$ be the  $i$th projection
map and the $i$th inclusion map canonically. If $R$ is right
$\aleph_{0}$-injective, for each $i$, suppose $f_{i} :
T_{i}\rightarrow R_{i}$ is $R_{i}$-linear where $T_{i}$ is a
countably generated right ideal of $R_{i}$. Then the map $0\times
\cdots\times T_{i}\times \cdots\times 0\rightarrow 0\times
\cdots\times R_{i}\times \cdots\times 0$ given by $(0, \cdots,t_{i},
\cdots,0)\longmapsto (0,\cdots, f _{i}(t_{i}), \cdots, 0)$ is
$R$-linear with $0\times \cdots\times T_{i}\times \cdots\times 0$ a
countably generated  right ideal of $R$. So it has the form $c\cdot$
where $c\in R$. Thus $f_{i}=\pi_{i}(c)\cdot$. Conversely, let
$\gamma : T\rightarrow R$ be $R$-linear, where $T$ is a countably
generated right ideal of $R$. For each $i\in I$, let $T_{i}=\{x\in
R_{i}~|~ \iota_{i}(x)\in T\}$. Since $T$ is countably generated,
 $T=\sum_{k=1}^{\infty}a_{k}R$, where $a_{k}\in R$, $k=1,2,\cdots$.
 Then it is easy to prove that $T_{i}=\sum_{k=1}^{\infty}\pi_{i}(a_{k})R_{i}$ is a
countably generated right ideal of $R_{i}$, $\forall i\in I$. Now
define $\gamma_{i}: T_{i}\rightarrow R_{i}$ by
$\gamma_{i}(x)=\pi_{i}\gamma(\iota_{i}(x))$, $x\in T_{i}$. Since
$R_{i}$ is right  $\aleph_{0}$-injective, there exists $c_{i}\in
R_{i}$ such that $\gamma_{i}=c_{i}\cdot$. For each $\bar t=\langle
t_{i}\rangle \in T$, write $\gamma(\bar t)=\bar s=\langle
s_{i}\rangle$. Since $T$ is a right ideal of $R$, $t_{i}\in T_{i},
\forall i\in I$. Thus $s_{i}=\pi_{i}(\bar
s\cdot\iota_{i}(1_{i}))=\pi_{i}(\gamma(\bar
t)\cdot\iota_{i}(1_{i}))=\pi_{i}\gamma(\bar
t\cdot\iota_{i}(1_{i}))=\pi_{i}\gamma(\iota_{i}(t_{i}))=\gamma_{i}(t_{i})=c_{i}t_{i}$,
whence $\bar s=\langle c_{i}\rangle\cdot\bar t $. So $\gamma=\langle
c_{i}\rangle\cdot$. This shows that $R$ is right
$\aleph_{0}$-injective.

\end{proof}

\begin{prop}\label{prop 2}
If R is right $\aleph_{0}$-injective, then {\bf l}{\rm(}$I\cap
K${\rm)}={\bf l}{\rm(}I{\rm)}+{\bf l}{\rm(}K{\rm)}, where I and K
are countably generated  right ideals of R.
\end{prop}
\begin{proof}
It is only to be shown that {\bf l}{\rm(}$I\cap
K${\rm)}$\subseteq${\bf l}{\rm(}$I${\rm)}+{\bf l}{\rm(}$K${\rm)}.
Let $x\in{\bf l}(I\cap K)$. Define a right  $R$-homomorphism $f$
from $I+K$ to $R_{R}$ such that $f(i+k)=xi$, where $i\in I$ and
$k\in K$. Then it is clear that $f$ is well-defined. Since  $I$ and
$K$ are both countably generated right ideals of $R$, $I+K$ is also
a countably generated right ideal of $R$. As $R$ is right
$\aleph_{0}$-injective, $f$ can be extended to a homomorphism from
$R_{R}$ to $R_{R}$. Hence there exist an element $c\in R$ such that
$f=c\cdot$. Thus, by the definition of $f$, $c\in {\bf l} (K)$ and
$x-c\in {\bf l} (I)$. So $x=(x-c)+c\in {\bf l} (I)+{\bf l} (K)$.
\end{proof}
Recall that a ring $R$ is called right {\it Kasch} if each simple
right $R$-module can embed into $R_{R}$. Or equivalently, every
maximal right ideal of $R$ is a right annihilator. Left Kasch rings
can be defined similarly.
\begin{prop}\label{prop 3}
If R is right Kasch and right $\aleph_{0}$-injective, then every
countably generated right ideals of R is a right annihilator.
\end{prop}
\begin{proof}
Let $I$ be a countably generated right ideal of $R$. If $I$ is not a
right annihilator, then there exists a nonzero element $x\in R$ such
that $x\in {\bf r}{\bf l}(I)\backslash I$. Now let $K=I+xR$. Then
$\overline{K}=K/I$ is finitely generated. Hence $\overline{K}$ has a
maximal submodule $\overline{M}$. Since $R$ is right Kasch,
$\overline{K}/\overline{M}$ can embed into $R_{R}$. Thus there
exists a homomorphism $f$ from $K$ to $R_{R}$ with $f(I)=0$ and
$f(x)\neq 0$. Since $R$ is right $\aleph_{0}$-injective, $f=c\cdot$
for some $c\in R$. So $c\in {\bf l}(I)$. Since $x\in {\bf r}{\bf
l}(I)$, $f(x)=cx=0$. This is a contradiction.
\end{proof}
\begin{thm}\label{thm 4}
Let R be a right $\aleph_{0}$-injective ring. For any idempotent
$e\in R$ with ReR=R, the corner ring eRe is also right
$\aleph_{0}$-injective.
\end{thm}
\begin{proof} Let $S=eRe$ and $\theta : T\rightarrow S$ be a right
$S$-homomorphism from a countably generated right ideal  $T$ of $S$
to $S_{S}$. Define $\bar{\theta}:TR\rightarrow R_{R}$ by
$\bar{\theta}(\sum t_{i}r_{i})=\sum \theta (t_{i})r_{i}$, $t_{i}\in
T$, $r_{i}\in R$. Assume $\sum t_{i}r_{i}=0$. For any $r\in R$,
$0=\sum t_{i}r_{i}re=\sum t_{i}(er_{i}re)$. So $0=\sum \theta
(t_{i})(er_{i}re)=[\sum \theta (t_{i})r_{i}]re$. Since $ReR=R$,
 it is clear that $\sum \theta (t_{i})r_{i}=0$. Hence $\bar{\theta}$ is a well-defined
right $R$-homomorphism. Since $T$ is a countably generated right
ideal of $S$, $TR$ is also a countably generated right ideal of $R$.
As $R$ is right $\aleph_{0}$-injective, $\bar{\theta}=c\cdot$ for
some $c\in R$. Then for each $ t\in T$,
$\theta(t)=e\theta(t)=e\bar\theta(t)=ect=(ec)et=(ece)t$. Hence
$\theta=(ece)\cdot$, as required.
\end{proof}
\begin{rem}\label{rem 5}
  {\rm The condition that $ReR=R$ in the
above theorem is necessary. For example (see \cite[Example 9]{K95}),
let $R$ be the algebra of matrices over a field $K$ of the form
{$R=\left[
\begin{array}{ccccccc}
a & x & 0 & 0& 0 & 0 \\
0 & b & 0 & 0& 0 & 0 \\
0 & 0 & c & y& 0 & 0 \\
0 & 0 & 0 & a& 0 & 0 \\
0 & 0 & 0 & 0& b & z \\
0 & 0 & 0 & 0& 0 & c \\
\end{array}
\right]$, $a,b,c,x,y,z\in K$.} \\ Set
$e=e_{11}+e_{22}+e_{44}+e_{55}$, which is a sum of canonical matrix
units. It is clear that $e$ is an idempotent of $R$ such that
$ReR\neq R$.  $R$ is right $\aleph_{0}$-injective, but $eRe$ is not
right $\aleph_{0}$-injective.}
\end{rem}

\begin{proof} \cite[Example 9]{K95} shows that $R$ is a $QF$ ring and $eRe$
is not a $QF$ ring. Since $R$ is $QF$, $R$ is right self-injective
and $eRe$ is left noetherian. So $R$ is right
$\aleph_{0}$-injective. If $eRe$ is right $\aleph_{0}$-injective,
then $eRe$ is $QF$ by Theorem 14. This is a contradiction.
\end{proof}
It is natural to ask whether right $\aleph_{0}$-injectivity is a
Morita invariant.
\begin{que}\label{que 6}
If R is right $\aleph_{0}$-injective, is M$_{n\times n}$(R) ($n\geq
2$) right $\aleph_{0}$-injective?
\end{que}
The method in the proof of the following theorem is owing to
\cite[Theorem 1]{NPY00}
\begin{thm}\label{th 7}
The following are equivalent for a ring R and an integer
$n\geq1$:\\
{\rm (1)} M$_{n}${\rm(}R{\rm)} is right $\aleph_{0}$-injective.\\
{\rm (2)} For each countably generated right R-submodule T of
R$_{n}$, every R-linear map $\gamma$:
T$\rightarrow$ R can be extended to R$_{n}$$\rightarrow$ R.\\
{\rm (3)} For each countably generated right R-submodule T of
R$_{n}$, every R-linear map $\gamma$: T$\rightarrow$ R$_{n}$ can be
extended to R$_{n}$$\rightarrow$ R$_{n}$.
\end{thm}
\begin{proof}
We prove for the case $n$=2. The others are analogous. \\
(1)$\Rightarrow$(2).\\ Given $\gamma$: $T\rightarrow R$ where $T$ is
a countably generated right $R$-submodule of $R_{2}$, consider the
countably generated right ideal $\overline T$=$[T~
T]=\{[\alpha~\beta]|\alpha, \beta\in T\}$ of $M_{2}(R)$. The map
$\overline\gamma$ : $\overline T$$\rightarrow$$M_{2}(R)$ defined by
\begin{center}
$\overline\gamma([\alpha~\beta])$=$\left[
                         \begin{array}{cc}
                            \gamma(\alpha)&\gamma(\beta)\\
                            0&0
                            \end{array}\right], \alpha, \beta\in T$ \end{center}is $M_{2}(R)$-linear.
                            By (1), there exists $C\in M_{2}(R)$
such that $\overline\gamma=C\cdot$. So $\gamma= \alpha\cdot$, where
$\alpha$ is the first row of
$C$. Hence $\gamma$ can be extended to a homomorphism from $R_{2}$ to $R$.\\
(2)$\Rightarrow$(3).\\ Given (2), consider $\gamma$: $T\rightarrow
R_{2}$ where $T$ is a countably generated right $R$-submodule of
$R_{2}$. Let $\pi_{i}: R_{2}\rightarrow R$ be the $i$th projection,
$i=1,2$. Then (2) provides an $R$-linear map $\gamma_{i}:
R_{2}\rightarrow R$ extending $\pi_{i}\circ\gamma$, $i=1, 2$. Thus
$\overline\gamma:R_{2}\rightarrow R_{2}$ extends $\gamma$ where
$\overline\gamma(\overline x)=[\gamma_{1}(\overline x)~
\gamma_{2}(\overline x)]^{T}$,
 $\overline x\in R_{2}$.\\
(3)$\Rightarrow$(1).\\
 Write $S=M_{2}(R)$, consider $\gamma$:
$T\rightarrow S_{S}$ where $T$ is a countably generated right ideal
of $S$. Then it is easy to prove that $T$=$[T_{0}~ T_{0}]$ where
$T_{0}=\{\overline x\in R_{2} ~|~ [\overline x ~0]\in T\}$ is a
right countably generated right $R$-submodule of $R_{2}$. For
$\overline x\in T_{0}$, the $S$-linearity of $\gamma$ shows that
$\gamma[\overline x~0]=\gamma([\overline x~0]\left[\begin{array}{cc}
                            1&0\\
                            0&0
                            \end{array}\right])=\gamma([\overline
                            x~0])\left[\begin{array}{cc}
                            1&0\\
                            0&0
                            \end{array}\right]
                            =[\overline y~0]$ for
some $\overline y\in R_{2}$. Writing $\overline
y=\gamma_{0}(\overline x)$ yields an $R$-linear map $\gamma_{0}$:
$T_{0}\rightarrow R_{2}$ such that $\gamma[\overline
x~0]=[\gamma_{0}(\overline x)~0]$, $\overline x\in T_{0}$. Then
$\gamma_{0}$ extends to an $R$-linear map $\overline \gamma:
R_{2}\rightarrow R_{2}$ by (3). Hence $\gamma_{0}=C\cdot$ for some
$C\in S$. If $[\overline x~\overline y]\in T$ it follows that
$\gamma([\overline x~\overline y])=\gamma([\overline x~0]+[\overline
y~0]\left[\begin{array}{cc}
                            0&1\\
                            0&0
                            \end{array}\right])=[\gamma_{0}(\overline x)~0]+[\gamma_{0}(\overline y)~0]\left[
                         \begin{array}{cc}
                            0&1\\
                            0&0
                            \end{array}\right]=[C\overline x~C\overline y]=C[\overline x~\overline y]$. This shows
$\gamma=C\cdot$.
\end{proof}
Recall that a right ideal $L$ of a ring $R$ is called {\it small}
if, for any proper right ideal $L'$ of $R$, $L+L'\neq R_{R}$. Let
$I$ be a right ideal of $R$. $I$ is said to \emph{lie over} a direct
summand of $R_{R}$ if there exists an idempotent $e\in R$ such that
$I=eR\oplus (I\cap(1-e)R)$, where $I\cap(1-e)R$ is a small right
ideal of $R$.

\begin{lem}\label{lem 8}
\cite[Corollary 2.10]{N76} A ring R is semiperfect if and only if
every countably generated right ideal of R lies over a direct
summand of $R_{R}$.
\end{lem}
\begin{thm}\label{th 9}
Let R be semiperfect. If every homomorphism from a countably
generated small right ideal of R to $R_{R}$ can be extended to one
from $R_{R}$ to $R_{R}$, then R is right $\aleph_{0}$-injective.
\end{thm}
\begin{proof}
 Let $I$ be a countably generated right ideal of $R$ and $f$ be a
homomorphism from $I$ to $R_{R}$.  By the above lemma, $I=eR\oplus
K$, where $e$ is an idempotent of $R_{R}$ and $K=I\cap(1-e)R$ is a
small right ideal of $R$. Since $I$ is countably generated, $K$ is
also countably generated. By hypothesis, there exists a homomorphism
$g$ from $(1-e)R$ to $R_{R}$ such that $g|_{K}=f|_{K}$. For each
$x\in R$, define $F(x)=f(x_{1})+g(x_{2})$ where $x_{1}=ex$ and
$x_{2}=(1-e)x$. It is clear that $F|_{I}=f$.

\end{proof}
Now we turn to the main theorem of this note. First  look at some
lemmas.
\begin{lem}\label{lem 10}
If R is a left $\aleph_{0}$-injective ring with ACC on
right annihilators, then R is left finite dimensional.
\end{lem}
\begin{proof}
Assume $R$ is not left finite dimensional. Then there are nonzero
elements $a_{i}\in R,i=1,2,\ldots,$  such that
$\{Ra_{i}\}_{i=1}^{\infty}$ is an independent family of proper left
ideals of $R$. Let $I_{k}=\oplus_{i=k}^{\infty}Ra_{i},k=1,2,\ldots$.
Then {\bf r}($I_{1}$) $\subseteq${\bf r}($I_{2}$)$\subseteq\cdots$.
Since $R$ satisfies \emph{ACC} on right annihilators, there exists
$n\in \mathbb{N}$ such that ${\bf r}(I_{n})={\bf r}(I_{n+1})$. As
$I_{n}=I_{n+1}\oplus Ra_{n}$, we have {\bf r}($I_{n}$)={\bf
r}($I_{n+1}$)$\cap${\bf r}($a_{n}$). So  {\bf r}($I_{n}$)
$\subseteq${\bf r}($a_{n}$). Since $R$ is left
$\aleph_{0}$-injective, by the symmetry of Proposition 2, $R$={\bf
r}(0)={\bf r}($I_{n+1}\cap Ra_{n}$)= {\bf r}($I_{n+1}$)+{\bf
r}($a_{n}$)={\bf r}($I_{n}$)+{\bf r}($a_{n}$)={\bf r}($a_{n}$). Thus
$a_{n}=0$. This is a contradiction.
\end{proof}
Recall that a ring $R$ is called left \emph{P-injective}
(\emph{2-injective}) if every homomorphism from a principal
(2-generated) left ideal of $R$ to $_{R}R$ can be extended to one
from $_{R}R$ to $_{R}R$.
\begin{lem}\label{lem 11}
\cite[Theorem 3.3]{NY95}If $R$ is left $P$-injective and left finite
dimensional, then $R$ is semilocal.
\end{lem}
\begin{lem}\label{lem 12}
 \cite[Theorem 2.7]{GG98}If $R$ is right noetherian and left P-injective, then J
is nilpotent.
\end{lem}
\begin{lem}\label{lem 13}
\cite[Corollary 3]{R75}If $R$ is a left 2-injective ring with ACC on
left annihilators, then $R$ is QF.
\end{lem}
Now we obtain the main theorem.
\begin{thm}\label{th 14}
If $R$ is right noetherian and left $\aleph_{0}$-injective, then $R$
is QF. \end{thm}
\begin{proof}
Since $R$ is right noetherian, $R$ satisfies \emph{ACC} on right
annihilators. By Lemma 10, $R$ is left finite dimensional. Since $R$
is left $\aleph_{0}$-injective, $R$ is left $P$-injective. So $R$ is
semilocal and $J$ is nilpotent by Lemma 11 and Lemma 12. Thus $R$ is
semiprimary. Hence $R$ is right artinian. So $R$ satisfies
\emph{ACC} on left annihilators. Then $R$ is \emph{QF} by Lemma 13.
\end{proof}
By Lemma 13, we see that if $R$ is a left $\aleph_{0}$-injective
ring with \emph{ACC} on left annihilators, then $R$ is \emph{QF}. It
is natural to ask the following question:
\begin{que}\label{que 15}
Can right noetherian condition in the above theorem  be weakened to
the condition satisfying ACC on right annihilators?
\end{que}
\begin{rem}\label{rem 16}
The answer is ''yes" if we can show that $J$ is right $T$-nilpotent.
By Lemma 10 and Lemma 11, $R$ is semilocal. If $J$ is right
$T$-nilpotent, then $R$ is right perfect. So $R$ is left GPF (i.e.,
$R$ is left $P$-injective, semiperfect and
$S_{l}\subseteq^{ess}$$_{R}R$). Thus $R$ is left Kasch by
\cite[Theorem 5.31]{NY03}. By \cite[Lemma 2.2]{NY95}, $R$
 is right $P$-injective. So $R$ is left and right mininjective. Recall that a ring $R$ is called right $mininjective$ if
  every homomorphism from a minimal right ideal of $R$ to $R_{R}$ can be extended to one from $R_{R}$
   to $R_{R}$. Left mininjective rings can be defined similarly.
 Then  by \cite[Theorem 2.5]{SC06}, $R$ is \emph{QF}.
\end{rem}
\bigskip
\begin{center}
{ACKNOWLEDGEMENTS}
 \end{center}
\indent\indent The research is supported by the Foundation of
Southeast University \\(No.4007011034).

\end{document}